\documentclass{amsart}
\usepackage{amsmath} 
\usepackage{amscd}
\usepackage{amsthm}
\usepackage{amsbsy}
\usepackage[margin=1in]{geometry}
\usepackage{commath, longtable}
\usepackage{comment, enumerate}
\usepackage{geometry}
\usepackage[matrix,arrow]{xy}
\usepackage{mathrsfs}
\usepackage{color}
\usepackage{float}
\usepackage{mathtools,caption}
\usepackage[table,dvipsnames]{xcolor}
\usepackage{tikz-cd}
\usepackage{longtable}
\usepackage[utf8]{inputenc}
\usepackage[OT2,T1]{fontenc}
\usepackage[normalem]{ulem}
\usepackage{hyperref}
\usepackage[customcolors]{hf-tikz}
\usepackage{enumitem}
\newlist{primenumerate}{enumerate}{1}
\setlist[primenumerate,1]{label={\arabic*$'$}}
\hypersetup{
 colorlinks=true,
 linkcolor=RoyalBlue,
 filecolor=blue,
 citecolor=Bittersweet,
 urlcolor=orange,
 pdftitle={Jha-Kundu-Laxmi-Washington project},
 }
\usepackage{booktabs}
\usepackage{latexsym,amssymb}
\usepackage{amssymb}
\usepackage{graphicx}
\usepackage{ textcomp }
\usepackage{stmaryrd}
\usepackage{tikz}
\usepackage{tikzpagenodes} 
\usepackage{multicol}
\usepackage[english]{babel} 
\usepackage{blindtext}
\usepackage[new]{old-arrows}

\definecolor{grn}{rgb}{0,0.6,0}
\definecolor{mrn}{rgb}{0.3,0,0}
\definecolor{blue}{rgb}{0,0,0.7}
\definecolor{Mygray}{rgb}{0.75,0.75,0.75}
\definecolor{auburn}{rgb}{0.43, 0.21, 0.1}
\definecolor{britishracinggreen}{rgb}{0.0, 0.26, 0.15}
\definecolor{taupe}{rgb}{0.28, 0.24, 0.2}

\newtheorem{theorem}{Theorem}[section]
\newtheorem*{theorem*}{Theorem}
\newtheorem{propn}[theorem]{Proposition}
\newtheorem{cor}[theorem]{Corollary}

\newtheorem{lthm}{Theorem}

\newtheorem{heur}[theorem]{Heuristic}
\newtheorem{lheur}{Heuristic}

\theoremstyle{remark}
\newtheorem{rmk}[theorem]{Remark}

\newcommand{\Nc}{\mathcal{N}}
\newcommand{\Nm}{\mathsf{N}}
\newcommand{\Z}{\mathbb{Z}}
\newcommand{\Q}{\mathbb{Q}}
\newcommand{\F}{\mathbb{F}}

\renewcommand{\exp}{\operatorname{exp}}

\newcommand{\Gal}{\operatorname{Gal}}
\newcommand{\rank}{\operatorname{rk}}

\newcommand{\Cl}{\operatorname{Cl}}
\newcommand{\Wt}{\operatorname{Wt}}
\newcommand{\Aut}{\operatorname{Aut}}

\title[]{$p$-class groups in  the cyclotomic $\Z_p$-extension of imaginary quadratic fields}

\author[S.~Jha]{Somnath Jha}
\address[Jha]{Department of Mathematics, Indian Institute of Technology Kanpur, Uttar Pradesh, India}
\email{jhasom@iitk.ac.in}

\author[D.~Kundu]{Debanjana Kundu}
\address[Kundu]{Department of Mathematics \& Statistics\\ University of Regina, Saskatchewan, Canada}
\email{debanjana.kundu@uregina.ca}

\author[H.~Laxmi]{H.~Laxmi}
\address[Laxmi]{Department of Mathematics, Indian Institute of Technology Kanpur, Uttar Pradesh, India}
\email{hlaxmi@iitk.ac.in, 67laxmi9@gmail.com}

\author[L.~C.~Washington]{Lawrence C.~Washington}
\address[Washington]{Department of Mathematics\\
University of Maryland, College Park, Maryland, USA}
\email{lcw@umd.edu}

\begin{document}
\baselineskip=14.5pt



\renewcommand{\thefootnote}{\arabic{footnote}}
\setcounter{footnote}{0}

\begin{abstract}
We study the structure of the $p$-class group of the first layer of the cyclotomic $\Z_p$-extension of an imaginary quadratic field where $p$ is non-split.
We obtain restrictions on the growth of the cyclic factors from the base layer to the first layer; in particular, when the $p$-rank remains unchanged, each cyclic factor grows by exactly one power of $p$.
We show that in a large number of cases the Iwasawa $\lambda$-invariant is completely determined by the $p$-rank of the first layer itself.
Finally we use a Cohen--Lenstra--Martinet type philosophy to study the probability distribution of Iwasawa lambda invariants of the first layer when the base layer has cyclic $p$-class group.
\end{abstract}

\maketitle

\section{Introduction}

\subsection{Background}
A main goal in classical Iwasawa theory is to study the asymptotic growth of the $p$-primary parts of class groups in a $\Z_p$-extension of a number field.
For the cyclotomic $\Z_p$-extension of an abelian number field (over $\Q$), the vanishing of the $\mu$-invariant was proved in \cite{ferrero-washington}.
In particular, for the imaginary quadratic fields considered in this paper, the eventual growth of the $p$-class groups in the cyclotomic tower is governed by the $\lambda$-invariant.
There is considerable interest in understanding how much of this asymptotic behaviour is already visible in the initial layers.
For example, J.~Sands \cite{sands} studied small Iwasawa invariants for imaginary quadratic fields and obtained explicit information on class-number growth in the first layers in situations where $\lambda$ is small.

Recently, the second and the fourth-named author \cite{kw_rms} initiated a systematic study of the structure of the $p$-class group at the first layer of the cyclotomic $\Z_p$-extension of an imaginary quadratic field $k_0$.
Assuming that $p$ is non-split and $A_0 = A(k_0)$ is cyclic, they determined the possible structures of $A_1 = A(k_1)$ and used this description to compare the heuristics of Cohen--Lenstra--Martinet for class groups with their naive heuristics for the Iwasawa $\lambda$-invariant (see Section~8 of \textit{loc.~cit.}).
This suggests several natural questions beyond the cyclic case: how does the structure of $A_0$ constrain that of $A_1$?
How are changes in the $p$-rank reflected in the growth of the invariant factors and exponent?
To what extent can the $p$-rank of the first-layer determine the cyclotomic Iwasawa $\lambda$-invariant?
Do the resulting first-layer distributions remain compatible with the existing heuristics for $\lambda$ given in \cite{EJV} or the naive heuristics in \cite{kw_rms}?

In this paper, we show that many of these questions are governed by the mod-$p$ Galois-module structure of the first-layer class group.
More precisely, the Jordan decomposition of $A_1/pA_1$ under the action of $\Gal(k_1/k_0)$ often controls the growth of the $p$-rank, the exponent, and the cyclic factors of $A_1$, and in a large range determines the Iwasawa $\lambda$-invariant.

\subsection{Main Results and Proof Ideas}

We highlight the main results we prove in this paper.
Throughout this paper, $k_0$ is an imaginary quadratic field in which $p$ is non-split.
Some of the results we prove are stronger than what is recorded in this introduction to avoid notational clutter.

\subsubsection{Structural Results}
We begin by recording the main structural consequences of the mod-$p$ Galois action on the first-layer class group.

\begin{lthm}[Rank Growth]
\label{ThmA}
Let $k_0$ be an imaginary quadratic field.
Let $p$ be an odd prime that is non-split in $k_0$.
Let $A_0$ (resp.~$A_1$) be the $p$-Sylow subgroup of the class group of $k_0$ (resp.~$k_1$).
Let $c$ denote the dimension of the subspace of $A_1/pA_1$ generated by the images of classes lifted from $A_0$.
Then 
\[
\rank_p A_0 +(p-1)c \le \rank_p A_1  \le p \rank_p A_0.
\]
In particular, at most $\lfloor (\rank_p A_1 - \rank_p A_0)/(p-1)\rfloor$ cyclic factors of $A_0$ fail to grow.
\end{lthm} 

Another interpretation of the quantity $c$ in Theorem~\ref{ThmA} that becomes more apparent from the proof is
\[
c = \#\{\text{Jordan blocks of } A_1/pA_1 \text{ of length }p\}.
\]

Theorem A has many nice consequences, all under the same assumptions as Theorem A.
Let 
\[
A_0 \simeq \bigoplus_{i=1}^{r_0} \Z/p^{e_i}\Z \text{ where } e_1\geq e_2 \geq \ldots \geq e_{r_0} \geq 1.
\]

\begin{lthm}\emph{[Theorem~\ref{thm 4.7}]}
If $\rank_p A_0 = \rank_p A_1=r_0$, then $A_1 \cong \bigoplus_{i}^{r_0} \Z/p^{e_i +1}\Z$.
\end{lthm}
In fact, by \cite[Theorem~1(b)]{fukuda}, when $\rank_p A_0 = \rank_p A_1$, the rank stabilizes in the cyclotomic tower.
By a Nakayama's Lemma argument, this rank equals the Iwasawa $\lambda$-invariant.
So we see the predicted growth in the first level of the $\Z_p$-extension.

\begin{lthm}\emph{[Theorem~\ref{rank lambda equal no cyclic hypothesis}]}
If $\rank_{p} A_1 <p$ then $\rank_p A_1 = \lambda$.
\end{lthm}

\begin{lthm}\emph{[Corollary~\ref{cor 5-7}]}
If $A_0$ is cyclic, then $\rank_p A_1 = \min\{p, \lambda\}$.
\end{lthm}

These last two results show that we can often determine the $\lambda$-invariant from $A_1$.
In Section~\ref{sec: 6}, we use these results to give heuristics for the $\lambda$-invariant.

Since $\abs{A_n}\to \infty$ as $n\to\infty$ but the Iwasawa $\mu$-invariant vanishes, the exponent of $A_n$ eventually increases with $n$.
The following results show that the exponent must increase from $A_0$ to $A_1$, except in special cases.

\begin{lthm}\emph{[Theorem~\ref{prop 4.8} and Corollary~\ref{3-4 strong}]}
Let $s = \{i \ : e_i = e_1\}$.
Write $\exp(\cdot)$ to denote the exponent of a group.
If $\rank_p A_1 < \rank_p A_0 +s(p-1)$ then $\exp(A_1)> \exp(A_0)$.
In particular, suppose that $s=\rank_p A_0 = r_0$ (in other words, suppose that $A_0 \simeq (\Z/p^a\Z)^{r_0}$ and $\exp(A_1)= \exp(A_0) =p^a$).
Then $\rank_p A_1 =p r_0$.
\end{lthm}

\begin{lthm} 
With notation as before, $\exp(A_0) = \exp(A_1)=p$ if and only if $\rank_p(A_1)=p\cdot \rank_p(A_0)$.
\end{lthm}

The main idea in Theorem~\ref{ThmA} and its consequences is to study the mod-$p$ Galois-module structure of the first-layer class group.
Let $G=\Gal(k_1/k_0)$.
The quotient $A_1/pA_1$ is naturally a $\F_p[G]$-module, and its Jordan-block decomposition records how classes from $A_0$ behave after lifting to $k_1$.
The lengths of the blocks control the growth of the $p$-rank; also, they impose restrictions on the growth of the exponent and of the individual cyclic factors.
This yields the rank bounds in Theorem~\ref{ThmA} and the structural consequences.

\subsubsection{Heuristics}
These structural results also provide the bridge to the heuristic questions considered in the second part of the paper.
This identification between first-layer rank and $\lambda$-invariant allows a direct comparison between first-layer class-group heuristics and the existing heuristics for $\lambda$.

In the cyclic case, the two distributions agree exactly.
The first-layer rank determines the $\lambda$-invariant as long as the rank is $<p$, while rank $p$ corresponds exactly to the case $\lambda\ge p$.
Consequently, a heuristic distribution for $\rank_pA_1$ can be compared directly with the heuristic distribution for $\lambda$.
Using the philosophy of Cohen--Lenstra--Martinet heuristics and the concrete description of the possible structure of $p$-class groups of $k_1$ when $A_0$ is cyclic (from \cite{kw_rms}) we derive the following.

\begin{lheur}[Heuristic Consequences]
\label{ThmC}
Fix an odd prime $p$.
Let $k_0$ denote an imaginary quadratic field and $k_1$ to denote the first layer of the cyclotomic $\Z_p$-extension of $k_0$.
Then
\[
\Pr(\rank_p A_1 = d \mid \rank_p A_0=1) = \begin{cases}
    \frac{p-1}{p^d} & \text{ if } 1\leq d \leq p-1 \\
    \frac{1}{p^{p-1}} & \text{ if } d= p\\
    0 & \text{ if } d>p.
\end{cases}.
\]
\end{lheur}

\begin{lheur}
With the same assumptions as in the previous heuristic, let $\lambda$ be the Iwasawa invariant for the cyclotomic $\mathbb Z_p$-extension of $k_0$.
Then
\begin{align*}
\Pr(\lambda = d \mid \rank_p A_0=1) &=    \frac{p-1}{p^d} \quad \text{ if } 1\leq d \leq p-1 \\
\Pr(\lambda \ge p \mid \rank_p A_0 =1) &=    \frac{1}{p^{p-1}}.
\end{align*}
\end{lheur}

It is interesting that the heuristic for the size of $\lambda$ agrees with what we expect for the first coefficient of a random $p$-adic power series, to be not divisible by $p$,  where the probability of $p$ dividing a given coefficient (other than the constant term) is $1/p$.
This is the basis of the naive heuristic that was proposed in \cite{kw_rms}.

\section*{Acknowledgements}
This project was initiated while DK was in residence at the Lodha Mathematical Sciences Institute, Mumbai.
She is supported by NSERC Discovery grants RGPIN-2026-07384 and DGECR-2026-00254. The third author acknowledges the Indian Institute of Technology Kanpur for the Institute Fellowship and facilities that support this research. 

\section{Preliminaries}
\subsection{Notation and Setup}
\label{sec: setup}
\begin{itemize}
\item $k_0$ is an imaginary quadratic field.
\item $p$ is an odd prime that is not split in $k_0$
\item $k_1$ is the first level of the cyclotomic $\mathbb Z_p$-extension of $k_0$.
\item $G=\langle \sigma\rangle = \Gal(k_1/k_0)$, generated by an element $\sigma$ with $\sigma^p=1$.
\item $A_i$ is the Sylow $p$-subgroup of the ideal class group of $k_i$.
\item $r_i=\rank_p (A_i)$, where the $p$-rank of a finite abelian $p$-group $A$ is defined to be $\rank_p(A)=\dim_{\mathbb F_p}A/pA$.
\item $\exp(A)$ is the exponent of a finite abelian group $A$; that is,  the largest order of an element of $A$.
\item $j: A_0\to A_1$ is the natural lifting map.
It is known in our situation that $j$ is an  injection (see \cite[Proposition~13.26]{washington_book}).
\item $\mathcal{N}: A_1\to A_0$ is the norm map on ideal class groups.
Since $k_1/k_0$ is totally ramified, $\mathcal{N}$ is surjective.
\item $\Nm=1+\sigma+\cdots +\sigma^{p-1}\in \mathbb Z[G]$.
For example, the map $\Nm: A_1\to A_1$ is the same as $j\circ \mathcal{N}$.
\item For a finite abelian $p$-group $A$, we let $\overline{A}=A/pA$.
\item $\overline{j}: \overline{A}_0\to \overline{A}_1$ is the natural map induced by $j: A_0\to A_1$.
\item We use the decompositions
\begin{equation}
\label{A0 decomp}
A_0\simeq \bigoplus_{i=1}^{r_0} \mathbb Z/p^{e_i}\mathbb Z, \qquad A_1\simeq \bigoplus_{i=1}^{r_1} \mathbb Z/p^{f_i}\mathbb Z,
\end{equation}
with $e_1\ge e_2 \ge \dots$ and $f_1\ge f_2\ge \dots$.
\end{itemize}

\subsection{Basic Results}
\label{sec: 2}

Let $F$ be a number field and $p$ be a prime number.
Let $K/F$ be an extension of number fields.
Write $\Cl(*)$ to denote the class group of a number field.
We denote the action of $\sigma \in {\Gal}(K/F)$ on $x \in \Cl(K)$ by $\sigma(x)$.
When $K/F$ is a cyclic extension, Chevalley \cite{chevalley} derived formula for the number of classes of $\Cl(K)$ fixed by the Galois action.

\begin{theorem}
\label{genusfor}
Let $K/F$ be a cyclic extension of number fields with Galois group $G$.
Let $\Cl(K)^{G}$ be the subgroup of $\Cl(K)$ consisting of the ideal classes that are fixed by the action of $G$ on $\Cl(K)$.
Let $\Nm_{K/F}$ denote the norm map from $K$ to $F$.
Let $E(F)$ be the unit group of $F$.
Then
\[
\#\Cl(K)^{G} =  \dfrac{\#\Cl(F)\prod_{v}e_v}{[K:F]\cdot\left[ E(F):E(F)\cap \Nm_{K/F}(K^{\times}) \right]}.
\]
where $v$ runs over all primes of $F$ and $e_v$ is the ramification index of $v$ of $F$ in $K$.
\end{theorem}

\begin{cor} Let $p$ be an odd prime, let $k$ be an imaginary quadratic field, and let $L/k$ be a cyclic extension of degree $p$ that is ramified at exactly one prime of $k$.
Then 
$\#\Cl(L)^G = \#\Cl(k)$.
\end{cor}

\begin{proof} Suppose $k\ne\mathbb Q(\sqrt{-3})$.
Since $-1=\Nm_{L/k}(-1)$ and  $\pm i=\Nm_{L/k}(i)$, we have
\[
\left[ E(k):E(k)\cap \Nm_{L/k}(L^{\times}) \right]=1.
\]
When $k=\mathbb Q(\sqrt{-3})$,  the right side of Chevalley's formula is
$1/[ E(k):E(k)\cap \Nm_{L/k}(L^{\times})]$, so the index is again equal to 1.
The result follows easily.
\end{proof}

\begin{propn}
\label{Ker N = A1^{sigma -1}}
Let the notation be as in Section \ref{sec: setup}
Then 
\begin{enumerate}
\item [\textup{(}a\textup{)}] $j(A_0) \cong A_1^{\langle \sigma \rangle},$
\item[\textup{(}b\textup{)}] $(\sigma - 1)A_1  := \{ (\sigma - 1)x \,\mid\,  \ x \in A_1\} = \ker(\Nc)$.
\end{enumerate}
\end{propn}

\begin{proof}
We know that
\[
\# A_1^{\langle \sigma \rangle} = \# A_0.
\]
Since $\#j(A_0)=\#A_0$ and $j(A_0)$ is contained in $A_1^{\langle \sigma \rangle}$, part (a) follows.

The exact sequence $0\to A_1^{\langle \sigma \rangle} \to A_1 \to (\sigma -1)A_1\to 0$ and part (a) tell us that 
\[
\#A_0= \# A_1^{\langle \sigma \rangle} = \# A_1/(\sigma -1)A_1.
\]
Since $(\sigma-1)A_1\subseteq \ker\mathcal{N}$ and $A_1/\ker\mathcal{N}\simeq A_0$, we must have $\ker(\mathcal{N})=(\sigma-1)A_1$.
\end{proof}

\section{Structure of the First-Layer Class Group}
\label{sec: 3}
We begin by proving an elementary result allowing us to compare the order of the groups $A_0$ and $A_1$.
Computations immediately indicate that $A_1$ is larger than $A_0$.
In fact, we have the following.

\begin{propn}
\label{3-1}
With notation introduced in Section~\ref{sec: setup},
\[
\#A_1 \geq p^{r_0} \cdot \#A_0.
\]
\end{propn}

\begin{proof}
We know that $j: A_0 \hookrightarrow A_1$ is injective and $\Nc: A_1 \to A_0$ is surjective.
We also know that $\Nc \circ j = p$,
so the kernel of $\Nc$ contains $j(A_0[p])$.
In particular, $\#\ker(\Nc)\geq p^{r_0}$ where $r_0$ is the $p$-rank of $A_0$.
We have
\[
\frac{\# A_1}{\# A_0} = \# \ker(\Nc) \geq p^{r_0}. \qedhere
\]
\end{proof}
Section~\ref{sec: relative growth}, we will greatly strengthen this result.

\subsection{Jordan Decomposition and Rank Growth}
We begin by analyzing the $\F_p[G]$-module structure of $A_1/pA_1$, which gives direct bounds on the growth of the $p$-rank from $A_0$ to $A_1$.

\begin{theorem}
\label{thm 1 LCW write-up}
With notation introduced in Section~\ref{sec: setup},
\[
r_0 +(p-1)\rank_p(\overline{j}(\overline{A}_0)) \le r_1  \le p\cdot  r_0.
\]
\end{theorem}

\begin{proof}
We deduce from Proposition~\ref{Ker N = A1^{sigma -1}} that
\[
\frac{A_1}{(\sigma -1)A_1} \simeq A_0.
\]
Let $M=A_1/(\sigma-1)A_1$.
We further have
\[
\frac{\overline{A}_1}{(\sigma-1)\overline{A}_1} \simeq \frac{A_1}{pA_1+(\sigma-1)A_1}\simeq \frac{M}{pM}\simeq \overline{A}_0.
\]
Therefore, $\rank_p (\overline{A}_1/(\sigma-1)\overline{A}_1) = \rank_p A_0 =r_0$.

Since $\sigma^p =1$, it is clear that $T^p = 0$ on $\overline{A}_1$ where $T = \sigma-1$.
In particular, $T$ is nilpotent and we can write $\overline{A}_1$ as a finite module over $\F_p[\Z/p\Z] \simeq \F_p[T]/\langle T^p \rangle$, hence over $\F_p[T]$.
Therefore, the structure theorem for modules over a PID tells us that
\begin{equation}
\label{modp}
\overline{A}_1 \simeq \bigoplus_{i=1}^{r_0} \frac{\mathbb F_p[T]}{\langle T^{n_i}\rangle},
\end{equation}
for some integers $n_i$ with $1\le n_i\le p$.
There are $r_0$ summands because $\overline{A}_1/T\overline{A}_1$ has $p$-rank $r_0$.
This immediately yields the inequality 
\begin{equation}
\label{R as sum of ni}
r_1 = \dim_{\F_p} (\overline{A}_1) = \rank_p A_1 = \sum_{i=1}^{r_0} n_i \leq pr_0.
\end{equation}

Set $\Nm = 1 + \sigma + \ldots + \sigma^{p-1}$.
The map 
\[
\overline{A}_1 \longrightarrow \frac{\overline{A}_1}{(\sigma-1)\overline{A}_1} \longrightarrow \overline{A}_0 \longrightarrow \overline{j}(\overline{A}_0)\subseteq \overline{A}_1
\]
is the map from $\overline{A}_1$ to $\overline{A}_1$ given by 
\[
1+\sigma+\sigma^2+\cdots+\sigma^{p-1}= \sum_{i=1}^{p} \binom{p}{i}T^{i-1}\equiv T^{p-1} \pmod p.
\]
In other words, $T^p=T \overline{\Nm} = 0$ on $\overline{A}$ and $\overline{\Nm} = T^{p-1}$ on $\overline{A}$.
On a Jordan block of the form $V_n = \F_p[T]/\langle T^n \rangle$, the operator $T^{p-1}$ has one-dimensional image if $n=p$ and is zero otherwise.
Since $\Nm(A_1) = j(A_0)$, we have
\begin{equation}
\label{13}
\rank_p(\overline{j}(\overline{A}_0))=\rank_p( \operatorname{image}(T^{p-1}))=\#\{i \, | \, n_i=p\}.
\end{equation}
Set 
\[
c = \# \{i \mid n_i =p\}.
\]
Exactly $c$ of the Jordan blocks have length $p$, while each of the remaining $r_0-c$ blocks has length $\geq 1$, so 
\[
r_1=\sum_{i=1}^{r_0} n_i \ge cp + (r_0 - c) = r_0 + (p-1)c
\]
This completes the proof.
\end{proof}

\begin{cor}
\label{propn 4.7}
If $r_1<r_0 + p-1$ then
\[
j(A_0) \subseteq pA_1.
\]
\end{cor}

\begin{proof}
The theorem implies that $\overline{j}(\overline{A}_0)=0$.
This is equivalent to $j(A_0) \subseteq pA_1$.
\end{proof}

\begin{cor}
\label{Cor 2 from LCW write up}
Suppose that $r_1 < r_0 + p - 1$.
Then $f_i\ge e_i+1$ for $1\le i\le r_0$.
\end{cor}

\begin{proof} Since 
\[
pA_1\simeq \mathbb Z/p^{f_1-1}\mathbb Z\oplus \mathbb Z/p^{f_2-1}\mathbb Z\oplus \cdots \oplus \mathbb Z/p^{f_{r_1-1}}\mathbb Z
\]
contains $j(A_0)\simeq A_0$ as a subgroup, the result follows.
\end{proof}

The conclusion of Corollary~\ref{Cor 2 from LCW write up} occurs frequently in our computations.
\begin{itemize}
\item For $k_0=\mathbb Q(\sqrt{-302821})$, we have $A_0=\left(\mathbb Z/3\mathbb Z\right)^2$ and 
$A_1=\left(\mathbb Z/9\mathbb Z\right)^2$.
\item For $k_0=\mathbb Q(\sqrt{-303160})$, we have $A_0=\left(\mathbb Z/3\mathbb Z\right)^2$ and 
$A_1=\left(\mathbb Z/27\mathbb Z\right)^2\oplus \left(\mathbb Z/9\mathbb Z\right)$.
\item For $k_0=\mathbb Q(\sqrt{-306937})$, we have $A_0=\left(\mathbb Z/3\mathbb Z\right)^2$ and 
$A_1=\left(\mathbb Z/9\mathbb Z\right)^3$.
\item For $k_0=\mathbb Q(\sqrt{-303388})$, we have $A_0=\left(\mathbb Z/3\mathbb Z\right)^2$ and 
$A_1=\left(\mathbb Z/9\mathbb Z\right)\oplus \left(\mathbb Z/3\mathbb Z\right)^3$.
\end{itemize}

The last example shows that the result can fail if $\rank_p A_1 = \rank_p A_0+p-1$.

\begin{rmk}
\label{[3,3] to [9,3,3]}
{\it Corollary~\ref{Cor 2 from LCW write up} explains certain non-occurrences observed computationally.
When $A_0\simeq (\Z/3\Z)^2$ then $A_1 \not\simeq \Z/3^e\Z \times \Z/3\Z \times \Z/3\Z$ for any $e\geq 2$.
Such a group has $\rank_p A_1=3<4=\rank_p A_0+(p-1)$, so Corollary~\ref{Cor 2 from LCW write up} would force both cyclic factors of $A_0$ to gain a power of $3$, which is impossible for this proposed structure.
In particular, it rules out $A_1\simeq \Z/9\Z \times \Z/3\Z \times \Z/3\Z$, which we had first observed computationally.}
\end{rmk}

We return to this phenomenon in Section~\ref{sec: relative growth}, where we study the growth of the cyclic factors more systematically and obtain a quantitative statement for arbitrary rank growth.

\subsection{Exponent Stability}
We now study the case where the exponents of $A_0$ and $A_1$ are equal.
In the specific case that $A_0$ is $p$-elementary of rank $r_0$, we prove a criterion for $A_1$ to be a $p$-elementary group.

By the decomposition of $A_0$ in \eqref{A0 decomp}, we have $\exp(A_0)=p^{m_1}$.
Usually, the exponent of $A_1$ is larger than the exponent of $A_0$.
The following allows us to treat the case where these exponents are equal.
In particular, we show in this case  that
\[
\rank_{p}(\overline{j}(\overline{A}_0)) \geq \text{ number of cyclic factors of } A_0 \text{ having maximal order.}
\]

\begin{theorem}
\label{prop 4.8}
Let $\exp(A_0) = p^a$ and $s:=\#\{j: e_j = a\}$.
If $\exp A_1 = \exp A_0$, then
\[
r_1 \geq  r_0 + s(p-1).
\]
\end{theorem}

\begin{proof}
Since $e_1 \geq e_2 \geq \ldots \geq e_{r_0} \geq 1$, 
\[
o(h_1) = \cdots = o(h_s) = p^a,
\]
where $h_i$ are generators of $A_0$ with order $o(h_i) = p^{e_i}$.

\noindent \textit{Claim:} The elements $\overline{j}(h_i)$, for $1\leq i \leq s$, are linearly independent in $\overline{A_1}$.

\noindent \textit{Justification:}
If not, then there exist $s$ integers $\mu_i$, not all zero mod $p$,  such that 
\[
\mu_1 j(h_1) + \cdots + \mu_s j(h_s) \in pA_1.
\]
Let
\[
x = \mu_1 h_1 + \cdots + \mu_s h_s \in A_0.
\]
Moreover $o(x) = p^a$ since some $\mu_i\not\equiv 0$ mod $p$.
Since $j(x)\in pA_1$, there exists an element $y\in A_1$ such that $j(x)=py$.
By the injectivity of $j$, we know that 
\[
o(j(x)) = o(py) = p^a.
\]
But this would mean that $y\in A_1$ has order $p^{a+1}$.
This contradicts the assumption that $\exp(A_1)=p^a$.
This completes the justification.

It is straightforward to conclude from \eqref{13} that
$c\geq s$.
The result follows from Theorem~\ref{thm 1 LCW write-up}.
\end{proof}

\begin{rmk}
{\it Here is an alternative proof:  $A_0$ contains a summand $S=\left(\mathbb Z/p^a\mathbb Z\right)^s$, which injects into $A_1$, both of which have exponent $p^a$.
Every cyclic subgroup of $j(S)$ of order $p^a$ is generated by an element of maximal order in $A_1$ and hence splits off as a direct summand; see \cite[Lemma~13.1.9]{judson2020abstract}.
Applying this result $s$ times shows that $j(S)$ is a direct summand of $A_1$.
Therefore, $\overline{j}(S/pS)$ is a direct summand of $\overline{A}_1$ that is contained in $\overline{j}(\overline{A}_0)$.
Therefore, $s=\rank_p(S/pS)\le \rank_p(\overline{j}(\overline{A}_0))$.
Theorem~\ref{thm 1 LCW write-up} now yields the result.}
\end{rmk}

\begin{rmk}
Theorem \ref{prop 4.8} yields a weaker version of Corollary~\ref{Cor 2 from LCW write up}: If $r_1<r_0+p-1$, then $\exp(A_1)>\exp(A_0)$.
\end{rmk}

\begin{cor}
\label{3-4 strong}
Suppose that $A_0 \simeq (\Z/p^a\Z)^{r_0}$.
If $\exp(A_1)= \exp(A_0) =p^a$ then $r_1 =r_{0}p$.
\end{cor}

\begin{proof}
This follows from Theorem~\ref{prop 4.8} with $s=r_0$ and the fact that $r_1 = \rank_p A_1 \leq r_0p$.
\end{proof}

\begin{rmk}
{\it The converse is false when $a>1$ even when $r_0=1$.
For example, when $k_0=\Q(\sqrt{-3001057})$, we have $A_0\simeq \Z/9\Z$ and $A_1\simeq \Z/27\Z \times \Z/3\Z \times \Z/3\Z$.}
\end{rmk}

When $a=1$, more can be said.

\begin{theorem}
\label{p elem}
Suppose that $A_0$ is an elementary $p$-group of rank $r_0$.
Then $A_1$ is an elementary $p$-group if and only if $\rank_p A_1 = r_0p$.
\end{theorem}

\begin{proof}
($\Rightarrow$)
This is Corollary~\ref{3-4 strong} with $a=1$.

($\Leftarrow$)
We know from \eqref{R as sum of ni} that 
\[
r_1 = \sum_{i=1}^{r_0} {n_i} = r_0p;
\]
thus each $n_i =p$.
We have observed previously that  $\overline{\Nm} = T^{p-1}$ on $\overline{A}_1$.
On each Jordan block $\F_p[T]/\langle T^p \rangle$, the image of $T^{p-1}$ is one-dimensional.
By \eqref{13} we know that $c=r_0$.
Thus $\dim_{\F_p} \overline{j}(\overline{A}_0) = r_0$.
Consider the natural surjective map
\[
(\Z/p\Z)^{r_0}  \simeq j(A_0) \longrightarrow \overline{j}(\overline{A}_0).
\]
By dimension counting, the map must be an isomorphism.
Thus the kernel is $j(A_0) \cap pA_1 =0$.

\textit{Claim: $pA_1$ is trivial, i.e., $A_1$ is an elementary $p$-group.}

\textit{Justification:}
Suppose that $pA_1$ is non-trivial.
Since multiplication by $p$ commutes with the $\sigma$ action, we know that $pA_1$ is stable under $\Gal(k_1/k_0) = \langle \sigma \rangle$; i.e., $\sigma(pA_1) = pA_1$.
Recall that (by orbit counting) when a $p$-group acts on a non-zero finite $p$-group, there is a non-trivial fixed point.
By definitions and Proposition~\ref{Ker N = A1^{sigma -1}} we know 
\[
(pA_1)^{\langle \sigma \rangle} = pA_1 \cap A_1^{\langle \sigma\rangle} = pA_1 \cap j(A_0) = 0,
\]
which is impossible.
This forces $pA_1$ to be trivial, as claimed.
\end{proof}

\begin{rmk}
{\it When $p=3$ and $k_0=\mathbb Q(\sqrt{-767755})$, we have $A_0\simeq \left(\mathbb Z/3\mathbb Z\right)^2$ and $A_1\simeq \left(\mathbb Z/3\mathbb Z\right)^6$.
We do not know any examples where $A_0\simeq \left(\mathbb Z/9\mathbb Z\right)^2$ and $A_1$ has exponent $9$.
Of course, such examples should be rare since  examples where $A_0$ has exponent 9 and rank 2 are rare, and combining this with the requirement that $A_1$ have rank 6 is very restrictive.}
\end{rmk}

\subsection{Growth of Cyclic Factors}
\label{sec: relative growth}

We attempt to understand the structure of $A_1$ as much as possible when there is $p$-rank growth from $k_0$ to $k_1$.
In particular, we show that at most $\lfloor (r_1-r_0)/(p-1) \rfloor$ many factors of $A_0$ fail to grow in $A_1$.

\begin{propn}
\label{prop growth cyclic factors}
Let $G, H$ be finite abelian $p$-groups with $H\subseteq G$.
Write
\[
H\simeq \bigoplus_{i=1}^r \mathbb Z/p^{e_i}\mathbb Z, \qquad G\simeq\bigoplus_{i=1}^s \mathbb Z/p^{f_i}\mathbb Z,
\]
with $e_1\ge e_2 \ge \dots$ and $f_1\ge f_2\ge \dots$.
Let $\overline{H}=H/pH$ and $\overline{G}=G/pG$, and let $\overline{j}: \overline{H}\to \overline{G}$.
Then 
\[
\rank_p (\operatorname{image}(\overline{j})) \ge \#\{ i\,\mid\, e_i=f_i\}.
\]
\end{propn}

\begin{proof}
Let $B=pG\cap H$.
Then $H/B\simeq (H+pG)/pG\simeq \operatorname{image}(\overline{j})$, so $|H|/|B|=p^c$, where $c=\rank_p( \operatorname{image}(\overline{j}))$.

Since $pH\subseteq B\subseteq H$,  
\[
B\simeq \bigoplus_{i=1}^r \mathbb Z/p^{e_i-\delta_i}\mathbb Z,
\]
with $0\le \delta_i\le 1$ for each $i$ and $\sum_i \delta_i=c$.
Moreover, $B\subseteq pG$, so $e_i-\delta_i \le  f_i-1$ for all $i$.
Because $e_i\le f_i$ for all $i$, we have
\[
c=\#\{i\,\mid\, \delta_i=1\} \ge \#\{i\,\mid\, e_i=f_i\}. \qedhere
\]
\end{proof}

\begin{cor}
With the notation introduced in Section~\ref{sec: setup}, 
\[
\#\{ i\,\mid\, 1\le i\le r_0, e_i=f_i\} \le \frac{r_1-r_0}{p-1}.
\]
\end{cor}

\begin{proof}
Theorem~\ref{thm 1 LCW write-up} implies that $\rank_p(\overline{j}(\overline{A}_0)) \le (r_1-r_0)/(p-1)$.
The result follows from Proposition~\ref{prop growth cyclic factors}.
\end{proof}

\begin{theorem} 
With the notation introduced in Section~\ref{sec: setup},
\[
\rank_p(A_1/j(A_0)) \ge r_1-c \ge \frac{(p-2)r_1+r_0}{p-1}.
\]
\end{theorem}
\begin{proof} Let $e_i=0$ for $r_0< i\le r_1$.
Then 
\[
\rank_p(A_1/j(A_0)) \ge \#\{i\,\mid\, 1\le i\le r_1, f_i\ne e_i\}\ge r_1 - \frac{r_1-r_0}{p-1} = \frac{(p-2)r_1+r_0}{p-1}. \qedhere
\]
\end{proof}

For example, when $r_1<r_0+p-1$, we find that $\rank_p(A_1/j(A_0)) > r_1-1$, so $\rank_p(A_1/j(A_0)) = r_1=\rank_p(A_1)$.
Therefore,  every cyclic factor of $A_0$ must increase in $A_1$, as we saw in Corollary~\ref{Cor 2 from LCW write up}.

\begin{cor}
With notation as above, $\#A_1 \geq p^{r_1-c} \cdot \#A_0$.
\end{cor}

Since $c\le (r_1-r_0)/(p-1)$, we have 
\[
r_1-c \ge ((p-2)r_1+r_0 + r_0)/(p-1) \ge r_0.
\]
This result generalizes  Proposition~\ref{3-1}

\section{Group structure of \texorpdfstring{$A_1$}{} when the rank is stable}
\label{sec: 4}

The main result we prove in this section is the following.

\begin{theorem}
\label{thm 4.7}
Keep the notation introduced in Section~\ref{sec: setup}.
Suppose that $r_1 = r_0 =r$.
Then
\[
A_1 \simeq \Z/p^{e_1 + 1}\Z \oplus \ldots \Z/p^{e_r + 1}\Z.
\]
\end{theorem}

The proof occupies the remainder of this section.

By the assumption of rank stability, we know $j(A_0[p])=A_1[p]$, so the Galois action on $A_1[p]$ is trivial.
We first show that this forces $\sigma$ to be congruent to the identity modulo $p$ in the endomorphism ring of $A_1$, and then deduce that the kernel of group-ring norm is exactly $A_1[p]$.
Since the ideal norm $A_1 \to A_0$ is surjective, this identifies $A_0$ with $A_1/A_1[p]$, from which the exact structure of $A_1$ follows.

\begin{propn}
\label{Prop 1}
Let $p$ be an odd prime and let $A$ be a finite abelian $p$-group.
Let $\sigma$ be an automorphism of $A$ with $\sigma^p =1$ and such that $\sigma=1$ on $A[p]$.
Then there exists an endomorphism $\psi$ of $A$ such that
\[
\sigma = 1+p \psi.
\]
\end{propn}

\begin{proof}
Write $A\simeq \bigoplus_{i=1}^r \Z/p^{g_i}\Z$ with $g _1 \geq \cdots \geq g_r\ge 1$.
Let
$x_i = (0, \ldots, 0, \overline{1}, 0, \ldots 0)^{\top}$ denote a generator of $A$ and note that $o(x_i) = p^{g_i}$.
Let $M=(m_{ij})$ denote the matrix corresponding to $\sigma -1$.
Then
\[
M(x_j) = \sum_{i=1}^r m_{ij}x_i
\]
We want to prove that every $m_{ij}$ is divisible by $p$, and that after dividing the matrix $M$ by $p$, we still get an endomorphism of $A$.

It is clear that $p^{g_j-1}x_j$ has order $p$ for each $j$.
Since $\sigma=1$ on $A[p]$, we observe that
\[
M(p^{g_j -1} x_j)= \sum_{i=1}^r p^{g_j-1} m_{ij}x_i =  0.
\] 
In particular, for each $i,j$ we must have that $p^{g_j-1} m_{ij}x_i=0$.
Since $o(x_i) = p^{g_i}$, it follows that $p^{g_i}\mid p^{g_j-1} m_{ij}$.
From this we can conclude that the $p$-adic valuation of $m_{ij}$ is as follows 
\begin{equation}
\label{1}
\nu_p(m_{ij})\geq \max\{0, g_i - g_j +1\}.
\end{equation}
In particular, the following three situations can arise:
\begin{equation}
\label{2}
\begin{cases}
\nu_p(m_{ij})\geq g_i - g_j +1 \geq 2 & \text{ if } g_i > g_j \\
p\mid m_{ij} & \text{ if } g_i = g_j \\
p\nmid m_{ij} & \text{ then } g_i < g_j.
\end{cases}
\end{equation}
\textit{Notice the direction of implication in the third case.}

Consider a pair $(i,j)$ such that $p\nmid m_{ij}$.
Then we know from \eqref{2} that $g_i < g_j$.
Suppose we choose $(i,j)$ to be such that the difference $d=g_j - g_i>0$ is minimal.

\noindent \textit{Claim:} For $2\leq k \leq p-1$, the following congruence condition holds
\[
(M^k)_{ij} \equiv 0 \pmod{p}.
\]

\noindent \textit{Justification:}
We write
\[
(M^k)_{ij} = \sum_{i_1, \ldots, i_{k-1}} m_{i i_{k-1}} m_{i_{k-1} i_{k-2}} \ldots m_{i_1 j}.
\]
(For the sake of consistency, we may write $i_k =i$ and $i_0=j$, when required.)
So a typical summand is 
\[
m_{i i_{k-1}} m_{i_{k-1} i_{k-2}} \ldots m_{i_1 j}.
\]
If this summand is not divisible by $p$, then each factor is not divisible by $p$.
In view of \eqref{2}, we have the following decreasing sequence of exponents
\[
g_j > g_{i_1} > \ldots > g_{i_{k-1}} > g_i.
\]
We assumed that $d=g_j - g_i$ is minimal; thus, such a decreasing sequence cannot exist.
Thus, each summand is divisible by $p$ and hence $(M^k)_{ij}$ is divisible by $p$.
This justifies the claim.

Next, we show that 
\begin{equation}
\label{3}
M^p \equiv 0 \pmod{p^2}.
\end{equation}
Therefore, $M^p\equiv 0 \pmod p$.
As before, we consider a summand
\[
S = m_{i_p i_{p-1}} m_{i_{p-1} i_{p-2}} \ldots m_{i_1 i_0}.
\]
Suppose that $p^2 \nmid S$.
Then $\nu_p(S)\leq 1$.
At most one factor arising in $S$ can be divisible by $p$.
Equivalently, at least $p-1$ of the other factors have $p$-valuation equal to $0$.
We know from \eqref{2} that every factor which is not divisible by $p$ forces the exponent to drop,
whereas a factor of $p$-valuation $1$ means that the exponent cannot increase.
Since $p\geq 3$, we get a contradiction since $g_j - g_i =d$ is minimal.

By assumption
\begin{equation}
\label{4}
(1+M)^p - 1 = pM + {p\choose 2}M^2 + \ldots + {p \choose p-1}M^{p-1} + M^p =0.
\end{equation}
For our choice of $(i,j)$, when $2\leq k \leq p-1$, we observe that
\[
p^2 \mid {p \choose k} (M^k)_{ij}.
\]
In addition, we also know that $p^2 \mid (M^p)_{ij}$.
By comparing with \eqref{4}, we conclude that $p\mid m_{ij}$.
This contradicts our choice of $(i,j)$.
In particular, for each pair $(i,j)$, we have shown that $p \mid m_{ij}$.

When $g_i > g_j$, the $(i,j)$-entry of $M$ has $p$-valuation $\geq g_i - g_j +1$.
Thus, $p^{g_i - g_j} \mid (m_{ij}/p)$.
On the other hand, when $g_i \leq g_j$, we have shown that $p\mid m_{ij}$.
Thus, the matrix $B=(m_{ij}/p)$ represents an endomorphism $\psi$ of $A$.
Indeed, an integer matrix $B=(b_{ij})$ represents an endomorphism of $A$ precisely when 
\[
p^{g_i - g_j}\mid b_{ij} \text{ for } g_i > g_j,
\]
(see \cite{hillar2007automorphisms}).
This follows from the fact that a homomorphism $\Z/p^{g_j}\Z \to \Z/p^{g_i}\Z$ sends $1 \mapsto b_{ij}$ and is well-defined precisely when $p^{g_i} \mid p^{g_j}b_{ij}$.

But $p\psi$ is nilpotent on $A_1$; hence $\sigma = 1+p\psi$ is invertible and is therefore an automorphism.
\end{proof}

\begin{propn}
\label{Prop 2}
Keep the notation and assumptions in Proposition~\ref{Prop 1}.
Let $\mathsf{N} = 1 + \sigma + \sigma^2 + \ldots + \sigma^{p-1}$.
Then the kernel of $\mathsf{N}$ is $A[p]$.
\end{propn}

\begin{proof}
We start by writing
\[
\mathsf{N} = 1 + \sigma + \sigma^2 + \ldots + \sigma^{p-1} = \sum_{t=0}^{p-1} (1+p\psi)^t.
\]
Since $(1+p\psi)^t \equiv 1+tp\psi \pmod{p^2}$ for all $t$, it follows immediately that
\[
\mathsf{N} \equiv p + p\psi \sum_{t=0}^{p-1} t \equiv p + p\psi \left( \frac{p(p-1)}{2}\right) \equiv p \pmod{p^2}.
\]
Equivalently, we write $\mathsf{N} = p(1+p\theta)$ where $\theta$ is an endomorphism of $A$.
But, the matrix corresponding to $1+p\theta$ has determinant $1\pmod{p}$; hence $1+p\theta$ is an automorphism.
Thus,
\[
\ker (\mathsf{N}) = \ker (p) = A[p]. \qedhere
\]
\end{proof}

We now have all the ingredients to prove the main result.

\begin{proof}[Proof of Theorem~\ref{thm 4.7}]
Recall that $j: A_0 \hookrightarrow A_1$ is the (injective) lifting map.
Further $j(A_0) = A_1^{\langle \sigma \rangle}$.
In particular, restricting to the $p$-torsion, we have
\[
j(A_0[p]) \subseteq A_1[p]^{\langle \sigma \rangle}.
\]
Here $\sigma$ is the generator of the cyclic group $\Gal(k_1/k_0)\simeq \Z/p\Z$.
But we are assuming that the rank at the base level and the first level are equal.
In particular,
\[
\# A_0[p] = \# A_1[p] = p^r.
\]
The injectivity of $j$ now forces that
\[
j(A_0[p]) = A_1[p].
\]
Since every element of $j(A_0[p])$ is fixed by $\sigma$, it follows that every element of $A_1[p]$ is also fixed by $\sigma$.
In particular, $\sigma=1$ on $A_1[p]$.

Now, we know that the ideal norm $\Nc: A_1 \to A_0$ is surjective.
Since $\Nm = j\circ \Nc$ and $j$ is injective we know that $\ker \Nm = \ker \Nc$.
By Proposition~\ref{Prop 2} we know that the kernel is equal to $A_1[p]$.
Since 
\[
A_0 \simeq A_1/ A_1[p] \simeq \bigoplus_{i=1}^r \Z/p^{f_i-1}\Z
\]
and the cyclic factors are unique, the proof is complete.
\end{proof}

\section{Determining \texorpdfstring{$\lambda$}{}-invariant from \texorpdfstring{$p$}{}-rank of \texorpdfstring{$A_1$}{}}
\label{sec: 5}

We now turn from the structure of the first-layer class group to the Iwasawa $\lambda$-invariant.
Since $\mu=0$ in the cyclotomic $\Z_p$-extension under consideration, the $p$-ranks of the class groups eventually stabilize and are bounded above by $\lambda$.
The natural question is how much of this asymptotic invariant is already visible at a finite layer.
We show that, as soon as the $p$-rank of $A_n$ is strictly smaller than $p^n$, the $p$-rank of $A_n$ is 
already equal to $\lambda$.
In particular, the first layer determines $\lambda$ if $\rank_p A_1<p$.

\begin{theorem}
\label{rank lambda equal no cyclic hypothesis}
Let $k_0$ be an imaginary quadratic field in which $p>2$ is non-split.
If $1 \leq \rank_{p} A_n < (r_0 -1)+ p^n$ then $\rank_p A_n = \lambda$.
\end{theorem}

\begin{proof}
Let $X={\varprojlim_n} A_n$ be the Iwasawa module of the cyclotomic $\mathbb Z_p$-extension of $k_0$.
Since $\mu (X)=0$, we know $\rank_{\Z_p} X=\lambda $.
Since $X$ has no non-zero finite submodules we know that $X$ is a free $\Z_p$-module and $X\simeq \Z_p^\lambda$.
Thus $\dim_{\F_p} X/pX = \lambda$.
We may view $\overline{X}$ as an $\F_p [T]$-module.
In particular
\[
\overline{X} \simeq \bigoplus_{i=1}^{r_0} \frac{\F_p [T] }{\langle T^{a_i}\rangle}.
\]
The direct sum indeed has $r_0$ components because $X/TX \simeq A_0$ and $\dim_{\F_p}(\overline{X}/T\overline{X}) = \rank_p A_0=r_0$.
We conclude that
\[
\dim_{\F_p}(\overline{X}) = \lambda = \sum_{i=1}^{r_0} {a_i}.
\]
Recall that $A_n \simeq X/((1+T)^{p^n}-1)X$ and that $(1+T)^{p^n}-1\equiv T^{p^n} \pmod p$.
In the following, the first isomorphism is standard and the second is a consequence of the aforementioned structure theorem
\[
A_n/pA_n \simeq \overline{X}/T^{p^n} \overline{X} \simeq \bigoplus_{i=1}^{r_0} \frac{\F_p [T]}{\langle T^{\min\{a_i, p^n\}}\rangle}.
\]
Thus, we note
\[
\rank_{p} A_n = \sum_{i=1}^{r_0} \min \{a_i, p^n\}.
\]
We now invoke our hypothesis that $r_n<r_0 + p^n - 1$.
It is immediate that each $a_i<p^n$ and 
\[
r_n = \dim_{\F_p} A_n = \sum_{i} a_i =\lambda. \qedhere
\]
\end{proof}

An immediate corollary of the above assertion is the following

\begin{cor}
\label{cor 5-7}
Suppose that $A_0$ is cyclic.
Then $\rank_p A_1 = \min\{\lambda,p\}$.
Consequently,
\begin{align*}
\rank_p A_1 = d  &\Longleftrightarrow \lambda=d, \text{ when } 1\leq d <p, \\
\rank_p A_1=p &\Longleftrightarrow  \lambda \ge p.
\end{align*}
\end{cor}

\section{Compatibility of Heuristics}
\label{sec: 6}

\subsection{Compatibility of Heuristics}
Using the Cohen--Lenstra--Martinet philosophy (as discussed in \cite{kw_rms}) we provide heuristics for $\Pr( \rank_p A_1 = d \mid \rank_p A_0 =1)$.
As in \cite{kw_rms}, we look at $A_1$ as a module over $\Z[\sigma]$ and weight each possible module $M$ by $\abs{\Aut_{\Z[\sigma]}(M)}^{-1}$.
Using this weighting, we obtain the following.

\begin{heur}
\label{heuristics gen}
Fix an odd prime $p$.
Let $k_0$ denote an imaginary quadratic field with $p$ non-split and let $k_1$ be the first layer of the cyclotomic $\Z_p$-extension of $k_0$.
Then
\[
\Pr(\rank_p A_1 = d \mid \rank_p A_0=1) = \begin{cases}
    \frac{p-1}{p^d} & \text{ if } 1\leq d \leq p-1 \\
    \frac{1}{p^{p-1}} & \text{ if } d= p\\
    0 & \text{ if } d>p.
\end{cases}
\]
\end{heur}

We give the derivation below.
Using Corollary~\ref{cor 5-7}, we immediately obtain the following.
\begin{heur}
With the same assumptions as in the previous heuristic, let $\lambda$ be the Iwasawa invariant for the cyclotomic $\mathbb Z_p$-extension of $k_0$.
Then
\begin{align*}
\Pr(\lambda = d \mid \rank_p A_0=1) &=    \frac{p-1}{p^d} \quad  \text{ if } 1\leq d \leq p-1 \\
\Pr(\lambda \ge p \mid \rank_p A_0 =1) &=    \frac{1}{p^{p-1}}.
\end{align*}
\end{heur}

Assume $A_0$ is non-trivial cyclic.
Let $F(T)$ be a characteristic power series of $X$.
For definiteness, we could choose $F$ to be the power series  that gives the $p$-adic $L$-function for the Dirichlet character of $k_0$ twisted by the Teichm\"uller character $\omega$ but this choice does not matter.
Then
\[
F(T)= a_0 + a_1T+a_2T^2+\cdots \in \Z\llbracket T\rrbracket.
\]
We know that $F(0)=a_0$ differs from the class number of $k_0$ by a $p$-adic unit, so $p\mid a_0$.
Also, $\lambda$ is the smallest $i$ such that $p\nmid a_i$.
Suppose we regard $F$ as a random power series where the probability is $1/p$, that $p$ divides a given coefficient $a_i$ with $i\ge 1$.
Then, when $1\le d\le p-1$, 
\[
\Pr(\lambda=d \, \mid \, \rank A_0 = 1) = \left(\frac{1}{p}\right)^{d-1}\left(\frac{p-1}{p}\right),
\]
and
\[
\Pr(\lambda\ge p \, \mid \, \rank A_0=1) = \frac{1}{p^{p-1}}.
\]
In other words, divisibility of the constant term $a_0$  by $p$ is determined by the Cohen--Lenstra heuristics.
But the other coefficients follow a more naive model.
The fact that the heuristics for the rank of $A_1$ matched the naive heuristics
for $\lambda$ when $A_0$ is cyclic is what led us to Theorem~\ref{rank lambda equal no cyclic hypothesis}.

The naive heuristics in \cite{kw_rms} differ from those of \cite{EJV} when the rank of $A_0$ is greater than 1.
It would be interesting to treat this situation and see which heuristic gives what appears to give the right answer.
However, analyzing rank-two modules is more difficult.

We now give the derivation of Heuristic~\ref{heuristics gen}, which occupies the remainder of the section.

Fix $A_0 \simeq \Z/p^m\Z$ where $m\geq 1$.
Then $A_1\simeq \Z[\sigma]/I$ for some ideal $I$ of $\Z[\sigma]$.
Each such $I$ has the form 
\[
I_{r,j}=((\sigma-1)^r+jp^{m-1}),\quad \text{ where } r\ge 1, 0\le j\le p-1,
\]
(see \cite[Section~6]{kw_rms}).
The ideals that occur in our situation are those with $j\ne 0$.
If we write $r=(p-1)s+t$ with $0\le t\le p-2$, we get the groups in Table~\ref{table 1}.
The congruence condition for $j$ in the table is modulo $p$.

\begin{center}
\begin{table}[h!]
\begin{tabular}{|c|c|c|}
\hline & group structure & $p$-rank \\
\hline
$m < s$ & $\left(\mathbb Z/{p^{m-1}}\right) \times \left(\mathbb Z/{p^{s+1}}\right)^{t+1}\times 
\left(\mathbb Z/{p^s}\right)^{p-t-2}$ & $p-1$ if $m=1$; $p$ if $m\ge 2$\\
\hline
$m>s$, $t\ne 0$ & $\left(\mathbb Z/{p^{m+1}}\right) \times \left(\mathbb Z/{p^{s+1}}\right)^{t-1}\times 
\left(\mathbb Z/{p^s}\right)^{p-t}$ & $t$ if $s=0$; $p$ if $s\ge 1$\\
\hline
$m>s$, $t=0$ & $\left(\mathbb Z/{p^{m+1}}\right) \times 
\left(\mathbb Z/{p^s}\right)^{p-2}\times \left(\mathbb Z/{p^{s-1}}\right)$ & $p-1$ if $s=1$; $p$ if $s\ge 2$\\
\hline
$m =s$, $t\ne 0$ & $\left(\mathbb Z/{p^{m-1}}\right) \times \left(\mathbb Z/{p^{m+1}}\right)^{t+1}\times 
\left(\mathbb Z/{p^m}\right)^{p-t-2}$ & $p-1$ if $m=1$; $p$ if $m\ge 2$\\
\hline
$m=s$, $t=0$, $j\not \equiv (-1)^m$ & $\left(\mathbb Z/{p^{m+1}}\right) \times \left(\mathbb Z/{p^{m-1}}\right) \times \left(\mathbb Z/{p^m}\right)^{p-2}$  & $p-1$ if $m=1$; $p$ if $m\ge 2$\\
\hline
$m=s$, $t=0$, $j\equiv (-1)^m$ & $  \left(\mathbb Z/{p^m}\right)^{p}$ &$p$\\
\hline
\end{tabular}
\medskip
\caption{The possible group structures for $A_1$.}
\label{table 1}
\end{table}
\end{center}
Let  $M_{r,j} = \Z_p[\sigma]/I_{r,j}$ where $r\geq 1$ and $1\leq j \leq p-1$.
It was proved in \cite{kw_rms} that $\# M_{r,j} = p^{m+r}$ and $\# \Aut_{\Z_p[\sigma]}(M_{r,j}) = (p-1)p^{m+r-1}$.
In view of the Cohen--Lenstra--Martinet philosophy, every pair $(r,j)$ has weight
\[
\Wt(M_{r,j}) = \frac{1}{\# \Aut_{\Z_p[\sigma]}(M_{r,j})} = \frac{1}{(p-1)p^{m+r-1}}.
\]
Since there are $p-1$ possible values of $j$ for each choice of $r$, the total weight corresponding to a fixed $r$ is given by
\[
\Wt^{(m)}_r = \sum_{j=1}^{p-1} \Wt(M_{r,j}) = \frac{1}{p^{m+r-1}}.
\]
Let 
\[
\mathcal{S}_m = \{M \mid M \text{ is a possibility for }A_1 \text{ when } A_0 \simeq \Z/p^m\Z\}.
\]
It was shown in \cite{kw_rms} that the total weight of all  possibilities for $A_1$ is
\begin{equation}
\label{all}
\Wt^{(m)}_{\operatorname{all}} = \sum_{M\in \mathcal{S}_m} \Wt(M) = \frac{1}{(p-1)p^{m-1}}.
\end{equation}
Thus, the probability attached to all $p-1$ ideals with given parameter $r$ is 
\begin{equation}
\label{parameter r}
\frac{\Wt^{(m)}_r}{\Wt^{(m)}_{\operatorname{all}}} = \frac{1/p^{m+r-1}}{1/(p-1)p^{m-1}} = \frac{p-1}{p^r}.
\end{equation}
Note that $\sum_{r\geq 1} (p-1)/p^r = 1$, as it should be.

For each $d$ and $m$, we want to compute the probability that $A_1$ has rank $d$, given that $A_0\simeq \Z/p^m\Z$.
This breaks down into cases.

{\bf Case I:} $1\le d\le p-2$

The only way this can happen is when $s=0$ and $t\ne 0$.
This is equivalent to $r=t$, so we obtain
\[
\Pr(\rank_p A_1 = d \mid A_0 \simeq \Z/p^m\Z)=\frac{\Wt^{(m)}_r}{\Wt^{(m)}_{\operatorname{all}}} = \frac{p-1}{p^d} \text{ when } 1\leq d \leq p-2.
\]
Note that this is independent of $m$, so we obtain
\[
\Pr(\rank_p A_1 = d \mid A_0 \text{ is cyclic }\ne 0) = \frac{p-1}{p^d}.
\]

{\bf Case II:} $d=p-1, m\ge 2$.

This happens when $s=1, t - 0$, so $r=p-1$, and 
\[
\Pr(\rank_p A_1 = p-1 \mid\ A_0 \text{ is cyclic }\ne 0) = \frac{p-1}{p^{p-1}}.
\]

{\bf Case III:} $d=p, m\ge 2$

We have 
\begin{align*}
\Pr(\rank_p A_1 = p \mid A_0 \simeq \Z/p^m\Z) & = 1 -  \sum_{d=1}^{p-1} \Pr(\rank_p A_1 = d \mid A_0 \simeq \Z/p^m\Z) \\
& = 1 - \sum_{d=1}^{p-1} \frac{p-1}{p^d}\\
& = 1 - \left( 1 - \frac{1}{p^{p-1}}\right) = \frac{1}{p^{p-1}}.
\end{align*}

{\bf Case IV:} $d=p, m=1$

This happens when $m=s=1, t=0, j\equiv -1$.

In this case, we emphasize that the contribution is from one value of $j$ only.
Thus,
\[
\Pr(\rank_p A_1 = p \mid A_0 \simeq \Z/p\Z) = \frac{\Wt(M_{p-1,p-1})}{\Wt^{(m)}_{\operatorname{all}}} = \frac{1/((p-1)p^{p-1})}{1/(p-1)} = \frac{1}{p^{p-1}}.
\]
 
{\bf Case V:} $d=p-1, m=1$

\begin{align*}
\Pr(\rank_p A_1 = p-1 \mid A_0 \simeq \Z/p\Z) & = 1 -  \sum_{d\ne p-1} \Pr(\rank_p A_1 = d \mid A_0 \simeq \Z/p\Z) \\
& = 1 - \sum_{d=1}^{p-2} \frac{p-1}{p^d} - \frac{1}{p^{p-1}}\\
& = 1 - \left( 1 - \frac{1}{p^{p-1}}\right) = \frac{1}{p^{p-1}}\\
&=\frac{p-1}{p^{p-1}}.
\end{align*}
The conditional probability we calculated in each case was independent of $m$, so averaging over the possible orders  $p^m$ of the non-trivial cyclic group $A_0$ does not change the answer.
This yields the desired heuristics.

\bibliographystyle{amsalpha}
\bibliography{reference}

@article{chevalley,
  title={Sur la th{\'e}orie du corps de classes dans les corps de nombres algebrique et dans les corps locaux},
  author={Chevalley, Claude},
  journal={J. Fac. Sci. Tokyo Imper. Univ., Sec. I-2},
  pages={363--476},
  year={1933}
}

@book{judson2020abstract,
  title={Abstract Algebra: Theory and Applications},
  author={Judson, Thomas W.},
  year={2025},
  publisher={Stephen F. Austin State University},
  note={https://judsonbooks.org/aata-files/aata-html/struct-section-finite-abelian-groups.html}
}

@article{ferrero-washington,
  title={The {I}wasawa invariant $\mu_p$ vanishes for abelian number fields},
  author={Ferrero, Bruce and Washington, Lawrence C},
  journal={Annals of Mathematics},
  volume={109},
  number={2},
  pages={377--395},
  year={1979},
  publisher={JSTOR}
}

@article{EJV,
  title={Modeling $\lambda$-invariants by $p$-adic random matrices},
  author={Ellenberg, Jordan S. and Jain, Sonal and Venkatesh, Akshay},
  journal={Comm. Pure Appl. Math.},
  volume={64},
  number={9},
  pages={1243--1262},
  year={2011},
  publisher={Wiley Online Library}
}

@article{hillar2007automorphisms,
  title={Automorphisms of finite abelian groups},
  author={Hillar, Christopher J. and Rhea, Darren L.},
  journal={The American Mathematical Monthly},
  volume={114},
  number={10},
  pages={917--923},
  year={2007},
  publisher={Taylor \& Francis}
}

@article{fukuda,
  title={Remarks on $\mathbb{Z}_p$-extensions of number fields},
  author={Fukuda, Takashi},
  journal={Proc. Japan Acad. Ser. A Math. Sci.},
  volume={70},
  pages={264--266},
  year={1994}
}

@article{kw_rms,
  title={The first level of $\mathbb{Z}_p$-extensions and compatibility of heuristics},
  author={Kundu, Debanjana and Washington, Lawrence C.},
  journal={Research in the Mathematical Sciences},
  volume={12},
  number={4},
  year={2025},
  publisher={Springer}
}

@article{sands,
  title={On small {I}wasawa invariants and imaginary quadratic fields},
  author={Sands, Jonathan W.},
  journal={Proceedings of the American Mathematical Society},
  volume={112},
  number={3},
  pages={671--684},
  year={1991}
}

@book {washington_book,
    AUTHOR = {Washington, Lawrence C.},
     TITLE = {Introduction to cyclotomic fields},
    SERIES = {Graduate Texts in Mathematics},
    VOLUME = {83},
   EDITION = {Second},
 PUBLISHER = {Springer-Verlag, New York},
      YEAR = {1997},
     PAGES = {xiv+487},
      ISBN = {0-387-94762-0},
       DOI = {10.1007/978-1-4612-1934-7},
       URL = {https://doi.org/10.1007/978-1-4612-1934-7},
}

\end{document}